\documentclass[11pt,a4paper,reqno]{amsart}
\usepackage{amsmath}
\usepackage{amssymb}
\usepackage{euscript}
\usepackage{pdfsync}
\usepackage[all]{xy}
\usepackage{color}
\usepackage{setspace}
\usepackage[colorlinks=true,linktocpage=true,pagebackref=false, citecolor=black,linkcolor=black]{hyperref}
\usepackage{cancel}
\usepackage[normalem]{ulem}
\usepackage{tikz,pgfplots}
\pgfplotsset{compat=1.10}
\usepgfplotslibrary{fillbetween}
\pgfplotsset{soldot/.style={color=black,only marks,mark=*}} \pgfplotsset{holdot/.style={color=black,fill=white,only marks,mark=*}}

\newtheorem{thm}{Theorem}[section]

\newtheorem{lem}[thm]{Lemma}
\newtheorem{prop}[thm]{Proposition}
\newtheorem{prodef}[thm]{Proposition and Definition}

\newtheorem{defnsprop}[thm]{Definitions and Proposition}

\newtheorem{cor}[thm]{Corollary}

\newtheorem{defn}[thm]{Definition}

\newtheorem{nots}[thm]{Notations}

\newtheorem{remark}[thm]{Remark}
\newtheorem{remarks}[thm]{Remarks}
\newtheorem{example}[thm]{Example}
\newtheorem{examples}[thm]{Examples}
\newtheorem{com}[thm]{Comment}

{\em}

{\em}

\newcounter{substep}
\def\thesubstep{\arabic{substep}}

{\em}

{}

 \newcommand{\N}{{\mathbb N}}
\newcommand{\Z}{{\mathbb Z}} \newcommand{\R}{{\mathbb R}}

 \newcommand{\BB}{{\mathbb B}}

 \newcommand{\J}{{\mathcal J}}

\newcommand{\gtp}{{\mathfrak p}} \newcommand{\gtq}{{\mathfrak q}}
\newcommand{\gtm}{{\mathfrak m}} \newcommand{\gtn}{{\mathfrak n}}
\newcommand{\gta}{{\mathfrak a}} \newcommand{\gtb}{{\mathfrak b}}

\newcommand{\Dd}{{\EuScript D}}
\newcommand{\Zz}{{\EuScript Z}}
\newcommand{\Uu}{{\EuScript U}}

\newcommand{\Pp}{{\EuScript P}}
\newcommand{\Ff}{{\EuScript F}}

\newcommand{\qf}{\operatorname{qf}}

\newcommand{\Int}{\operatorname{Int}}
\newcommand{\dist}{\operatorname{dist}}

\newcommand{\hgt}{\operatorname{ht}}

\newcommand{\Max}{\operatorname{Max}}
\newcommand{\Spec}{\operatorname{Spec}}

\newcommand{\cl}{\operatorname{Cl}}

\newcommand{\betas}{\operatorname{\beta_s\!}}
\newcommand{\betaa}{\operatorname{\beta_s^*\!\!}}

\newcommand{\ev}{\operatorname{ev}}

\newcommand{\Sat}{\operatorname{Sat}}
\newcommand{\Inter}{\operatorname{Inter}}
\newcommand{\ceros}{\operatorname{\mathcal Z}}

\newcommand{\x}{{\tt x}}  
 \renewcommand{\t}{{\tt t}}

\newcommand{\ol }{\overline}

\begin{document}

\title[Intermediate algebras of semialgebraic functions]{Intermediate algebras of semialgebraic functions}

\author{E. Baro}
\author{Jos\'e F. Fernando}
\author{J.M. Gamboa}

\address{Departamento de \'Algebra, Geometr\'\i a y Topolog\'\i a, Facultad de Ciencias Matem\'aticas, Universidad Complutense de Madrid, Plaza de Ciencias 3, 28040 MADRID (SPAIN)}

\email{eliasbaro@pdi.ucm.es}
\email{josefer@mat.ucm.es}
\email{jmgamboa@mat.ucm.es}

\thanks{Authors supported by Spanish STRANO PID2021-122752NB-I00 and Grupos UCM 910444}

\begin{abstract}
We characterize intermediate $\R$-algebras $A$ between the ring of semialgebraic functions ${\mathcal S}(X)$ and the ring ${\mathcal S}^*(X)$ of bounded semialgebraic functions on a semialgebraic set $X$ as rings of fractions of ${\mathcal S}(X)$. This allows us to compute the Krull dimension of $A$, the transcendence degree over $\R$ of the residue fields of $A$ and to obtain a \L ojasiewicz inequality and a Nullstellensatz for archimedean $\R$-algebras $A$. In addition we study intermediate $\R$-algebras generated by proper ideals and we prove an extension theorem for functions in such $\R$-algebras.
\end{abstract}

\maketitle


\maketitle
\setcounter{tocdepth}{2}
{\small
\begin{spacing}{0.01}
\tableofcontents
\end{spacing}
}

\section{Introduction}\label{s1}

A subset $X\subset\R^n$ is said to be \textit{basic semialgebraic} if it can be written as
$$
X:=\{x\in\R^n:\ f(x)=0,\, g_1(x)>0,\ldots,g_{\ell}(x)>0\}
$$
for some polynomials $f,g_1,\ldots,g_{\ell}\in\R[\x_1,\ldots,\x_n]$. The finite unions of basic semialgebraic sets are called \emph{semialgebraic sets}. 

A continuous function $f:X\to\R$ is said to be \emph{semialgebraic} if its graph is a semialgebraic subset of $\R^{n+1}$. Usually, semialgebraic function just means a function, non necessarily continuous, whose graph is semialgebraic. However, since all semialgebraic functions occurring in this article are continuous we will omit for simplicity the continuity condition when we refer to them. 

The sum and product of functions, defined pointwise, endow the set ${\mathcal S}(X)$ of semialgebraic functions on $X$ with a natural structure of commutative ring whose unity is the semialgebraic function ${\bf 1}_M$ with constant value 1. In fact ${\mathcal S}(X)$ is an $\R$-algebra, if we identify each real number $r$ with the constant function which just attains this value. The most simple examples of semialgebraic functions on $X$ are the restrictions to $X$ of polynomials in $n$ variables. Other relevant ones are the absolute value of a semialgebraic function, the distance function to a given semialgebraic set, the maximum and the minimum of a finite family of semialgebraic functions, the inverse and the $k$-root of a semialgebraic function whenever these operations are well-defined.

Let ${\mathcal S}^{*}(X)$ be the subring of bounded semialgebraic functions on $X$. It is indeed an $\R$-algebra and we will identify each real number $r\in\R$ with the constant function of value $r$. 

The study of rings of continuous functions has deserved a lot of attention from specialists in analysis, topology and algebra. The history of this theory is long and rich and its main development goes back to the 1950's and 1960's. This subject contributed in an important way to the appearance and evolution of well-known tools in Mathematics like the Stone-\v{C}ech compactification, the theory of nets and filters, the spectrum and the maximal spectrum of a commutative ring, \ldots We refer the reader to \cite{gj} for a very detailed study of rings of continuous functions. 

This article is inspired by the work \cite{dgm}, where the authors studied intermediate algebras between the $\R$-algebra ${\mathcal C}(X)$ of continuous real valued functions on a topological space $X$ and its sub-algebra ${\mathcal C}^{*}(X)$ of bounded functions, which has a precedent in the paper \cite{rs}. To work with continuous functions (on a completely regular topological space $X$) presents an advantage over the semialgebraic setting: the Stone-\v{C}ech compactification $\betas X$ is a topological space containing $X$ as a dense subspace, and so it makes sense to consider the $\R$-algebra ${\mathcal C}(\betas X)$. This is not so in the semialgebraic context. The so called semialgebraic Stone-\v{C}ech compactification $\betaa X$ of a semialgebraic set $X$ introduced in \cite{fg2} enjoys most of the properties of the classical Stone-\v{C}ech compactification; namely, it is a compact and Hausdorff topological space one of whose models is the set $\betaa X:=\Max({\mathcal S}^{*}(X))$ of maximal ideals of the ring ${\mathcal S}^{*}(X)$, the map
$$
{\tt j}:X\to\betaa X,\, x\mapsto\gtm^{*}_x:=\{f\in{\mathcal S}^{*}(X):\, f(x)=0\}
$$
is continuous, and ${\tt j}(X)$ is a dense subspace of $\betaa X$. This is why we identify $X$ with ${\tt j}(X)$. In addition, it is proved in \cite[4.4]{fg2} that for every $f\in{\mathcal S}^*(X)$ there exists a continuous function ${\widehat f}:\betaa X\to\R$ such that ${\widehat f}\circ{\tt j}=f$. But, in spite of its name, $\betaa X$ is very rarely a semialgebraic set, so it has no sense to consider the ring ${\mathcal S}(\betaa X)$. Indeed, it was proved in \cite{fg2} that $\betaa X$ is homeomorphic to a semialgebraic set if and only if $\betaa X\setminus X$ is finite.

On the other hand, the topology of semialgebraic sets is more friendly, which allows to achieve sharper results by using specific techniques of semialgebraic geometry. Although they are neither Noetherian nor enjoy primary decomposition properties, rings of semialgebraic functions are closer to polynomial rings than to classical rings of continuous functions. For example, the Lebesgue dimension of $\R$ is 1, see Problem 16F in \cite{gj},
whereas the Krull dimension of the ring ${\mathcal C}(\R)$ of real valued continuous functions on $\R$ is infinite, see Problem 14I \cite{gj}, and the Krull dimension of the ring ${\mathcal S}(\R)$ equals one, see \cite[Thm 1.1 and Thm. 1.2]{fg1}. Another simple but crucial example is the following. It was proved in \cite[2.2.8]{bcr} that the function
$$
f:X\to\R,\, x\mapsto\dist(x,Z):=\inf\, \{\|x-z\|:\, z\in Z\}
$$
is semialgebraic and so each closed semialgebraic subset $Z$ of $X$ satisfies $Z=f^{-1}(0)$. Thus, in the semialgebraic context, to be closed and to be the zeroset of a continuous semialgebraic function are the same thing. 

Indeed many tools of semialgebraic geometry has been succesfully employed in \cite{fg3}, \cite{fg6}, \cite{fg1}, \cite{fg2} and \cite{bfg} to obtain a lot of information about rings of semialgebraic functions. We will use the results obtained in these previous works.

Along this article we fix a semialgebraic set $X\subset\R^n$ and an intermediate $\R$-algebra $A$ containing ${\mathcal S}^*(X)$ and contained in ${\mathcal S}(X)$, and it is organized as follows. In Section \ref{s2}, we collect most of the preliminary definitions, notations and results that will be used freely in the sequel. Next, in Section \ref{s3} we construct in Theorem \ref{loc1} a natural bijection between the intermediate $\R$-algebras containing ${\mathcal S}^{*}(X)$ and contained in ${\mathcal S}(X)$ and the family of saturated multiplicatively closed subsets $D$ of ${\mathcal S}^*(X)$ such that $f^{-1}(0)=\varnothing$ for every $f\in D$. As a first consequence we prove in Corollary \ref{trdeg} that if $\kappa_{\gtp}:=\qf(A/\gtp)$ denotes the field of fractions of $A/\gtp$, where $\gtp$ is a prime ideal of $A$, then the transcendence degree over $\R$ of $\kappa_{\gtp}$ is finite and upperly bounded by $\dim (X)$. To finish Subsection \ref{krytrdeg}  we deduce from Theorem \ref{loc1} that the Krull dimension $\dim(A)$ of $A$ equals $\dim(X)$ in Proposition \ref{krull}. In Subsection \ref{s31} we show that the set of prime ideals containing a prime ideal of $A$ is a totally ordered set. This implies that $A$ is a Gelfand ring, that is, every prime ideal is contained in a unique maximal ideal of $A$. The main result is Theorem \ref{stone}, where we prove that $\Max(A)$ is a model of the semialgebraic Stone-\v{C}ech compactification of $X$. In Subsection \ref{s33} we study fixed and free ideals of $A$. We characterize fixed maximal ideals of $A$ in Theorem \ref{freefixed}. Next we introduce an unusual terminology. In Definition \ref{arquimedianO} we say that a maximal ideal $\gtm$ of $A$ is \em archimedean \em if the natural inmersion $\R\hookrightarrow A/\gtm$ is surjective. These ideals have been called  frequently \em real \em in the literature. But the adjective \em real \em has a different meaning in Real Algebraic Geometry: an ideal $\gta$ of an unitary and commutative ring $R$ is said to be \em real \em if whenever the elements $f_1,\dots,f_r\in R$ satisfy $f_1^2+\cdots+f_r^2\in\gta$, then each $f_i\in\gta$. The main results of this subsection are a \L ojasiewicz inequality \ref{lojasiewicz} and a Nulltellensatz \ref{radical} for intermediate $\R$-algebras whose maximal ideals are archimedean.

In Section \ref{s4} we study intermediate algebras generated by proper ideals. We prove first that we may assume that they are generated by $z$-ideals. Secondly, we prove in Theorem \ref{EXT} an extension result for functions in such $\R$-algebras.

\section{Preliminaries}\label{s2}

\subsection{Localization and Zariski spectrum}\label{localization1} A subset $D$ of a commutative and unitary ring $R$ is said to be \em multiplicatively closed \em if the unit element $1_R$ belongs to $D$, the zero element $0_R\notin D$ and given $d,e\in D$ then $ed\in D$. We denote $\Uu(R)$ the subset consisting on the units of $R$, that is, the invertible elements in $R$. The map $\varphi_D:R\to R_D,\, r\mapsto r\cdot1_R^{-1}$ is a homomorphism of unitary rings whose kernel is the intersection of $D$ with the set of zero divisors of $R$. Thus $\varphi_D$ is injective if and only if $D$ does not contain zero divisors of $R$. In addition, $\varphi_D$ is an isomorphism if and only if $D\subset\Uu(R)$.

The multiplicatively closed subset $D$ of $R$ is said to be \em saturated \em if whenever $x,y\in R$ satisfy $xy\in D$ then $x,y\in D$.

Let $\gta$ be an ideal of $R$.  Then, the \em extended ideal \em $\gta^{e}$ of $R_D$, that is, the smallest ideal of $R_D$ containing $\gta$, is the set $\gta^{e}:=\{ad^{-1}:\, a\in\gta, \, d\in D\}$. Hence, $\gta^{e}$ is a \em proper \em ideal of $R_D$, that is, $\gta^{e}\neq R_D$, if and only if $\gta\cap D=\varnothing$. On the other hand, let $\gtb$ be a proper ideal of $R_D$. Then, its \em contraction \em $\gtb^{c}:=\varphi_D^{-1}(\gtb)$ is an ideal of $R$ satisfying $\gtb^{c}\cap D=\varnothing$. In addition, $(\gtb^{c})^{e}=\gtb$. 

The set $\Spec(R)$ consisting of all prime ideals of $R$ is endowed with its Zariski topology, which has the family of subsets
$$
\Dd_R(f):=\{\gtp\in\Spec(R): f\notin\gtp\},\, \text{ where } \, f\in R,
$$
as a basis of open subsets. It is a compact space, see e.g. \cite[Ch. I, Ex. 17]{am}. In particular, the map
$$
\{\gtp\in\Spec(R):\gtp\cap D=\varnothing\}\to\Spec(R_D),\, \gtp\mapsto \gtp^{e}
$$
is a homeomorphism whose inverse is the map
$$
\Spec(R_D)\to\{\gtp\in\Spec(R):\gtp\cap D=\varnothing\},\, \gtq\mapsto\gtq^{c}.
$$ 
The \em maximal spectrum \em of $R$ is the subset $\Max(R)\subset\Spec(R)$ consisting of all maximal ideals of $R$. Indeed $\Max(R)$ is compact too because every covering of $\Max(R)$ by subsets of the form $\Dd_R(f)$ is a covering of $\Spec(R)$.
\subsection{Dimension of semialgebraic sets} The \em dimension \em $\dim(X)$ of a semialgebraic set $X\subset\R^n$ is defined as the dimension of the smallest algebraic set containing $X$. In other words, $\dim(X)$ is the Krull dimension of the quotient ring $\Pp(X):=\R[\x_1,\dots,\x_n]/{\mathcal J}(X)$, where
$$
{\mathcal J}(X):=\{f\in\R[\x_1,\dots,\x_n]: X\subset f^{-1}(0)\}.
$$
It is proved in \cite[2.8.10]{bcr} that given a semialgebraic set $X\subset\R^n$ and a point $x\in X$ there exists a semialgebraic neighbourhood  $U$ of $x$ in $X$ such that $\dim(U)=\dim(V)$ for every semialgebraic neighborhood $V\subset U$ of $x$. This common value is called the \em local dimension \em of $X$ at $x$.

\subsection{Semialgebraic Stone-\v{C}ech compactification of a semialgebraic set}\label{stone1} 

For every real valued function $f:X\to\R$ we denote $\Zz_X(f):=f^{-1}(0)$.

\noindent (1) Recall that a \em compactification \em of a topological space $Z$ is a pair $(K,{\tt j})$, where ${\tt j}:Z\to K$ is a continuous map and ${\tt j}(Z)$ is a dense subset of the compact space $K$. 

\noindent (2) Given two compactifications $(K_1,{\tt j}_1)$ and $(K_2,{\tt j}_2)$ of a topological space $Z$ it is said that $(K_2,{\tt j}_2)$ \em dominates \em $(K_1,{\tt j}_1)$, and we write $(K_1,{\tt j}_1)\preccurlyeq(K_2,{\tt j}_2)$, if there exists a continuous surjective map $\rho:K_2\to K_1$ such that $\rho\circ{\tt j}_2={\tt j}_1$. Note that since ${\tt j}_i(Z)$ is dense in $K_i$ for $i=1,2$, the map $\rho$ is unique satisfying the equality above. We say that $(K_1,{\tt j}_1)$ is \em smaller \em that $(K_2,{\tt j}_2)$.

\noindent (3) A compactification $(K,{\tt j})$ of a semialgebraic set $X$ is \em semialgebraically complete, \em see \cite[4.2]{fg2}, if for each $f\in{\mathcal S}^*(X)$ there exists a continuous function $F:K\to\R$ such that $f=F\circ{\tt j}$. Notice that $K$ is not necessarily a semialgebraic set. 

\noindent (4) It was introduced in \cite{fg2} the so called \em semialgebraic Stone-\v{C}ech compactification of a semialgebraic set $X$. \em It is a compact and Hausdorff topological space $\betaa X$ one of whose models is the set $\betaa X:=\Max({\mathcal S}^{*}(X))$ of maximal ideals of the ring ${\mathcal S}^{*}(X)$. The map
$$
{\tt j}:X\to\betaa X,\, x\mapsto\gtm^{*}_x:=\{f\in{\mathcal S}^{*}(X):\, f(x)=0\}
$$
is continuous and ${\tt j}(X)$ is a dense subspace of $\betaa X$. For simplicity we identify $X$ with ${\tt j}(X)$.  

In addition, it is proved in \cite[4.4]{fg2} that for every $f\in{\mathcal S}^*(X)$ there exists a continuous function ${\widehat f}:\betaa X\to\R$ such that ${\widehat f}\circ{\tt j}=f$. For every semialgebraic subset $Y$ of $X$ we denote $\cl_{\betaa X}(Y)$ the closure in $\betaa X$ of $Y$ and $\Int_X(Y)$ the interior of $Y$ in $X$. 

\noindent (5) It was proved in \cite[4.4.3]{fg2} that the semialgebraic Stone-\v{C}ech compactification $\betaa X$ of $X$ is the smallest among the semialgebraically complete compactifications of $X$.

\section{Intermediate $\R$-algebras as rings of fractions and consequences}\label{s3}

As announced in the Introduction we construct first a bijection between the intermediate $\R$-algebras containing ${\mathcal S}^{*}(X)$ and contained in ${\mathcal S}(X)$ and the family of saturated multiplicatively closed subsets $D$ of ${\mathcal S}^*(X)$ such that $\Zz_X(f)=\varnothing$ for every $f\in D$.

\subsection{Transcendence degree of the residual fields and Krull dimension}\label{krytrdeg}Let $D\subset{\mathcal S}^{*}(X)$ be a multiplicatively closed subset such that $\Zz_X(f)=\varnothing$ for every $f\in D$. Thus $D\subset\Uu({\mathcal S}(X))$ and 
$$
{\mathcal S}^{*}(X)\subset{\mathcal S}^{*}(X)_D\subset{\mathcal S}(X),
$$ 
i.e. $A:={\mathcal S}^{*}(X)_D$ is an intermediate $\R$-algebra between ${\mathcal S}^{*}(X)$ and ${\mathcal S}(X)$. 

In particular it was proved in \cite[3.2]{fg3} that ${\mathcal S}(X)={\mathcal S}^{*}(X)_D$, where $D$ denotes the multiplicatively closed subset of ${\mathcal S}^{*}(X)$ consisting on all $f\in{\mathcal S}^*(X)$ whose zeroset $\Zz_X(f)$ is empty, that is, all $f\in{\mathcal S}^*(X)$ that are invertible in ${\mathcal S}(X)$.

\begin{prop} \label{saturado} The following conditions hold:

\noindent \em (1) \em The set $D_A:=\Uu(A)\cap{\mathcal S}^*(X)$ is a saturated multiplicatively closed subset of ${\mathcal S}^*(X)$. 

\noindent \em (2) \em We have $A={\mathcal S}^*(X)_{D_A}$.
\end{prop}
\begin{proof} (1) The product of bounded semialgebraic functions is a bounded semialgebraic function too. In addition, the product of invertible elements of a unitary commutative ring is invertible too. Hence, $D_A$ is a multiplicatively closed subset of ${\mathcal S}^*(X)$. To see that it is saturated, let $f,g\in {\mathcal S}^*(X)$ such that $fg\in D_A$. In particular there exists $h\in A$ such that $(fg)\cdot h=1_A$. Therefore both $f$ and $g$ are units in $A$, so $f,g\in D_A$.

\noindent (2) The inclusion ${\mathcal S}^*(X)_{D_A}\subset A$ follows at once since $D_A\subset\Uu(A)$. Conversely, each $f\in A$ can be written as 
$$
f=(f\cdot(1+f^2)^{-1})/(1\cdot(1+f^2)^{-1})
$$
and $f\cdot(1+f^2)^{-1}\in{\mathcal S}^*(X)$, because its absolute value is upperly bounded by $1/2$, whereas $1\cdot(1+f^2)^{-1}\in D_A$ since its absolute value is upperly bounded by $1$, so it belongs to ${\mathcal S}^*(X)\subset A$, and it is a unit in $A$ whose inverse is $1+f^2\in A$. 
\end{proof}

Recall that a subset $D$ of a commutative and unitary ring $R$ is said to be \em multiplicatively closed \em if the unit element $1_R$ belongs to $D$, the zero element $0_R\notin D$ and given $d,e\in D$ then $ed\in D$. In addition, $D$ is said to be \em saturated \em if whenever $x,y\in R$ satisfy $xy\in D$ then $x,y\in D$.

\begin{thm}\label{loc1} Let $\Sat(X)$ be the family of saturated multiplicatively closed subsets $D$ of ${\mathcal S}^*(X)$ such that $\Zz_X(f)=\varnothing$ for every $f\in D$. Let ${\Inter}(X)$ be the family of intermediate $\R$-algebras between ${\mathcal S}^*(X)$ and ${\mathcal S}(X)$. Then the maps
$$
\Sat(X)\to{\Inter}(X),\, D\mapsto {\mathcal S}^*(X)_D\quad\&\quad{\Inter}(X)\to\Sat(X),\, A\mapsto D_A:=\Uu(A)\cap{\mathcal S}^*(X)
$$
are mutually inverse.
\end{thm}
\begin{proof} The result is the immediate consequence of Proposition \ref{saturado} and the fact that each multiplicatively closed subset $D$ of a unitary and commutative ring $R$ admites a \em saturation, \em which is the smallest saturated multiplicatively closed subset of $R$ containing $D$. Indeed, see \cite[Exercise 3.7]{am}, the complement $\ol{D}$ of the union of those prime ideals of $R$ not meeting $D$ is a saturation of $D$ and the map $R_D\to R_{\ol{D}},\, rd^{-1}\mapsto rd^{-1}$ is a bijection. 
\end{proof}

As a consequence of Thm. \ref{loc1} we get an upper bound of the transcendence degree over $\R$ of the residue fields $\kappa_{\gtp}$ for $\gtp$ a prime ideal of $A$.

\begin{cor}\label{trdeg} Let $\gtp$ be a prime ideal of $A$ and let $\kappa_{\gtp}:=\qf(A/\gtp)$ be the field of fractions of $A/\gtp$. Then, the transcendence degree over $\R$ of $\kappa_{\gtp}$ is finite and upperly bounded by $\dim X$.
\end{cor}
\begin{proof} Let $\gtq:=\gtp\cap{\mathcal S}^*(X)$. It was proved in \cite[Thm.1.3]{fg1} that the transcendence degree over $\R$ of $\kappa_{\gtq}:=\qf({\mathcal S}^*(X)/\gtq)$ is finite and upperly bounded by $\dim X$. Thus, all reduces to check that the field extension $\kappa_{\gtp}|\kappa_{\gtq}$ is algebraic. It suffices to see that each element of $A/\gtp$ is algebraic over $\kappa_{\gtq}$. Let $f\in A$ and $\xi:=p+\gtp$. With the notations in Thm. \ref{loc1} there exist $F,G\in{\mathcal S}^*(X)$ such that $f:=F/G$ where $G\notin D_A$, i.e., $G$ is not a unit in $A$. Note that $f$ is a root of the polynomial $G\cdot{\tt t}-F\in{\mathcal S}^*(X)[{\tt t}]$. Since $\gtp$ is a proper ideal of $A$ the intersection $\gtq\cap\Uu(A)$ is empty, so $G\notin\gtq$ and $\xi$ is a root of the polynomial $(G+\gtq)\cdot {\tt t}-(f+\gtq)\in\kappa_{\gtq}[{\tt t}]$.
\end{proof}

Next we apply Thm. \ref{loc1} to compute the Krull dimension of $A$. For each point $x\in X$ we denote $\gtm_{A;x}:=\{f\in A: f(x)=0\}$. It is a maximal ideal of $A$ because it is the kernel of the surjective homomorphism $A\to\R.\, f\mapsto f(x)$. We shall abreviate $\gtm^{*}:=\gtm_{{\mathcal S}^*(X);x}$.

\begin{prop}\label{krull} The Krull dimension $\dim(A)$ of $A$ equals $\dim(X)$. 
\end{prop}
\begin{proof} By Theorem \ref{loc1}, $A={\mathcal S}^*(X)_{D_A}$ where $D_A:=\Uu(A)\cap{\mathcal S}^*(X)$. In addition, the map 
$$
\Spec(A)\hookrightarrow\Spec({\mathcal S}^*(X)),\, \gtp\mapsto\gtp\cap{\mathcal S}^*(X) 
$$ 
induced by the inclusion of ${\mathcal S}^*(X)$ into $A$ is a homeomorphism onto its image, that is the set of prime ideals $\gtp$ of ${\mathcal S}^*(X)$ such that $\gtp\cap D_A=\varnothing$. Thus $\dim(A)\leq d:=\dim({\mathcal S}^*(X))=\dim(X)$, where the last equality was proved in \cite[Thm. 1.1]{fg1}.

To prove the equality notice that, by \cite[2.8.12]{bcr} there exists a point $x\in X$ such that the local dimension of $X$ at $x$ equals $d$. By \cite[Thm.1.2]{fg1} we have $\hgt(\gtm^{*}_x)=d$, where $\hgt(\gta)$ means the height of an ideal $\gta$. Let
$$
\gtp_0\subsetneq\gtp_1\subsetneq\cdots\subsetneq\gtp_d:=\gtm_x^{*}
$$
be a chain of prime ideals in ${\mathcal S}^*(X)$. Notice that $\Uu(A)\cap\gtm_x^{*}=\varnothing$. Otherwise pick $f\in\Uu(A)\cap\gtm_x^{*}$. Then $f(x)=0$ and there would exist $g\in A$ such that $f\cdot g=1_A$. Then $1=f(x)\cdot g(x)=0$, a contradiction. As the map 
$$
\{\gtp\in\Spec({\mathcal S}^*(X)):\gtp\cap D_A=\varnothing\}\to\Spec(A),\, \gtp\mapsto \gtp^{e}
$$
is injective and inclusion-preserving we have $\gtp_0^{e}\subsetneq\gtp_1^{e}\subsetneq\cdots\subsetneq\gtp_d^{e}:=\gtm_{A;x}$.Thus $d\leq\dim(A)$ and so $\dim(A)=\dim(X)$.
\end{proof}

\subsection{Main properties of the spectra of intermediate $\R$-algebras}\label{s31}

``Convexity'' is an ubiquitous condition in Real Geometry. This implies that the set of prime ideals containing a given prime ideal of $A$ form a chain and, in particular, $A$ is a Gelfand ring. Let us collect now the properties that will be useful in the sequel.

\begin{cor}[Convexity I]\label{convex1} \em (1) \em Every radical ideal $\gta\in{\mathcal S}^*(X)$ is \em absolutely convex, \em i.e, given $f,g\in{\mathcal S}^*(X)$ such that $|f(x)|\leq|g(x)|$ for every $x\in X$ and $g\in\gta$, then $f\in\gta$.

\noindent \em (2) \em A prime ideal $\gtp$ of $A$ is maximal if and only if $\gtq:=\gtp\cap{\mathcal S}^*(X)$ is a maximal ideal of ${\mathcal S}^*(X)$. 

\noindent \em (3) \em Every prime ideal $\gtp\in A$ is absolutely convex.

\noindent \em (4) \em The set of prime ideals of $A$ containing a prime ideal $\gtp$ of $A$ is a totally ordered set by inclusion. In particular $A$ is a Gelfand ring, i.e., each prime ideal of $A$ is contained in a unique maximal ideal of $A$.

\noindent \em (5) \em The subspace $\Max(A)$ of $\Spec(A)$ consisting on the maximal ideals of $A$ endowed with the Zariski topology is compact and Hausdorff. In addition, the map $\varphi:X\to\Max(A),\, x\mapsto\gtm_{A;x}$ is a bijection onto a dense subspace of $\Max(A)$. 
\end{cor}
\begin{proof} (1) Note that $|g|^2=g^2\in\gta$ because $g\in\gta$. As $f(x)^2=|f(x)|^2\leq|g(x)|^2$ for every $x\in X$ it follows from \cite[3.1.2]{fg3} that $f^2\in\gta$. Since $\gta$ is a radical ideal, it follows that $f\in\gta$. 

\noindent (2)  It follows from \cite[Exercise 3.21]{am} that the map
$$
{\tt i}_A:\Spec(A)\to\Spec({\mathcal S}^*(X)),\, \gtq\to\gtq\cap{\mathcal S}^*(X)
$$
is a homeomorphism onto its image ${\tt i}_A(A)$, that is the set of prime ideals $\gtq\in\Spec({\mathcal S}^*(X))$ such that $\gtq\cap D_A=\varnothing$. Its inverse is the map
$$
{\tt i}_A(\Spec(A))\to\Spec(A),\, \gtp\mapsto\gtp^{e}.
$$
Suppose that $\gtp$ is maximal but $\gtq$ is not maximal. Let $\gtm_A$ be a maximal ideal of $A$ containing properly the ideal $\gtq$. Then $\gtp=(\gtp^{c})^{e}=\gtq^{e}\subsetneq(\gtm_A^{c})^{e}=\gtm_A$, a contradiction. Conversely, suppose that $\gtq$ is maximal but $\gtp$ is not maximal. Let $\gtm$ be a maximal ideal of ${\mathcal S}^*(X)$ containing properly the ideal $\gtp$. Then, $\gtq=\gtp\cap{\mathcal S}^*(X)\subsetneq\gtm\cap{\mathcal S}^*(X)$ and $\gtq$ is not maximal.

\noindent (3) Let $f,g\in A$ such that $|f(x)|\leq|g(x)|$ for every $x\in X$ and $g\in\gtp$. Let $\gtq:=\gtp\cap{\mathcal S}^*(X)$. Note that the semialgebraic functions $(1+g^2)^{-1}$, $f\cdot(1+g^2)^{-1}$ and $g\cdot(1+g^2)^{-1}$ are bounded and
$$
\frac{|f(x)|}{1+g(x)^2}\leq\frac{|g(x)|}{1+g(x)^2}\, \text{ for every } x\in X.
$$
In addition $g\cdot(1+g^2)^{-1}\in\gtq$ because $g\in\gtp$. As $\gtq$ is a radical ideal of ${\mathcal S}^*(X)$ it is absolutely convex, by part (1). Thus $h:=f\cdot(1+g^2)^{-1}\in\gtq$. Therefore $f=(1+g^2)\cdot h\in\gtp$.

\noindent (4) Let $\gtp_1$ and $\gtp_2$ be two prime ideals of $A$ containing $\gtp$. Let $\gtq_i:=\gtp_i\cap{\mathcal S}^*(X)$ for $i=1,2$ and let $\gtq:=\gtp\cap{\mathcal S}^*(X)$. Both $\gtq_1$ and $\gtq_2$ are prime ideals of ${\mathcal S}^*(X)$ containing the prime ideal $\gtq$ of ${\mathcal S}^*(X)$. Since the set of prime ideals of ${\mathcal S}^*(X)$ containing $\gtq$ is totally ordered, see \cite[3.1.4]{fg3} we can assume that $\gtq_1\subset\gtq_2$, and so $\gtp_1=\gtq_1^{e}\subset\gtq_2^{e}=\gtp_2$.

\noindent (5) It was proved in \cite{dmo} that the maximal spectrum of a Gelfand ring is a Hausdorff space, whereas  we recalled in \ref {localization1} that $\Max(A)$ is a compact space. To prove that  $\varphi(X)$ is a dense subset of $\Max(A)$ let $f\in A$ be such that the basic open subset $\Dd_A(f)$ is nonempty. Thus $f\not\equiv0$ and it exists a point $x\in X$ such that $f(x)\neq0$, that is, $\gtm_{A;x}\in\Dd_A(f)$, i.e., $\Dd_A(f)\cap\varphi(X)\neq\varnothing$. Notice that $\gtm_{A;x}$ is a maximal ideal of $A$ because it is the kernel of the surjective homomorphism $A\to\R,\, f\mapsto f(x)$.

To finish we must show that $\varphi$ is injective. Let $x,y\in X$ with $x\neq y$. The bounded semialgebraic function
$$
f:X\to\R,\, z\mapsto \frac{\|z-x\|^2}{1+\|z-x\|^2}
$$
satisfies $f(x)=0$ and $f(y)\neq0$. Thus $f\in\gtm_{A;x}\setminus\gtm_{A;y}$, so $\gtm_{A;x}\neq\gtm_{A;y}$.
\end{proof}

As a consequence we show now that $\Max(A)$ is a model of the semialgebraic Stone-\v{C}ech compactification of $X$.

\begin{thm}\label{stone} Let $X$ be a semialgebraic set and let $A$ be an intermediate algebra between ${\mathcal S}^{*}(X)$ and ${\mathcal S}(X)$. The map $\rho_A:\Max(A)\to\betaa X$ that maps each maximal ideal $\gtn$ of $A$ to the unique maximal ideal of ${\mathcal S}^*(X)$ containing the prime ideal $\gtn\cap{\mathcal S}^*(X)$ of ${\mathcal S}^*(X)$ is a homeomorphism.
\end{thm}
\begin{proof} The map $\rho_A$ is the composition of the inmersions 
$$
\Max(A)\hookrightarrow\Spec(A)\, \text{  and }\, \Spec(A)\hookrightarrow\Spec({\mathcal S}^{*}(X)),\, \gtp\mapsto\gtp\cap{\mathcal S}^{*}(X)
$$
with the retraction $\rho_X^{*}:\Spec({\mathcal S}^{*}(X))\to\betaa X$ that maps each prime ideal of ${\mathcal S}^{*}(X)$ to the unique maximal ideal of ${\mathcal S}^{*}(X)$ containing it (recall that ${\mathcal S}^{*}(X)$ is a Gelfand ring). Indeed it follows from  \cite[1.2]{dmo} that $\rho_X^{*}$ is continuous and so $\rho_A$ is continuous too. Consequently, since $\Max(A)$ and $\betaa X$ are compact and Hausdorff, $\rho_A$ is a closed map, and this implies that it is surjective because its image is a closed subset of $\betaa X$ that contains the dense subspace $\{\gtm_x^{*}=\rho_A(\gtm_{A;x}):\, x\in X\}$, see \ref{stone1} (4). Thus all we need to prove is that $\rho_A$ is injective. Let $\gtm_1$ and $\gtm_2$ be two distinct maximal ideals of $A$. Then $\gtm_1+\gtm_2=A$ and there exist $f_1\in\gtm_1$ and $f_2\in\gtm_2$ such that $f_1+f_2=2$. In particular $\Zz_X(f_1^2+f_2^2)=\Zz_X(f_1)\cap\Zz_X(f_2)=\varnothing$. Let us see that 
\begin{equation}\label{suma}
\frac{|f_i(x)|}{f_1(x)^2+f_2(x)^2}\leq1\, \text{ for } \, i=1,2 \, \text{ and for every point } \, x\in X.
\end{equation}
It is enough to prove it for $i=1$. If $|f_1(x)|\leq1$ then $f_1(x)\leq1$ and $f_2(x)\geq1\geq|f_1(x)|$. If $|f_1(x)|\geq1$ then $f_1(x)^2+f_2(x)^2\geq f_1(x)^2\geq|f_1(x)|$. In both cases the inequality \eqref{suma} holds for $i=1$. Consequently, $g_i:=|f_i|\cdot(f_1^2+f_2^2)^{-1}\in{\mathcal S}^*(X)\subset A$. In addition, $|f_i|^2=f_i^2\in\gtm_i$, so $|f_i|\in\gtm_i$. Hence the function
$$
F_i=\frac{f_i^2}{f_1^2+f_2^2}=|f_i|\cdot g_i\in\gtm_i.
$$
Thus $F_1$ and $F_2$ are bounded semialgebraic functions and $F_i\in\gtm_i$ for $i=1,2$. Consequently $F_i\in\rho_A(\gtm_i)$, which implies that $\rho_A(\gtm_1)\neq\rho_A(\gtm_2)$ since $F_1+F_2=1$.
\end{proof} 





\subsection{Fixed and free ideals}\label{s33}

\begin{defn} \em  An ideal $\gta$ of $A$ is said to be \em fixed \em if all functions in $\gta$ vanish simultaneously at some point of $X$. Otherwise, the ideal $\gta$ is \em free\em.  
\end{defn}

Our first goal is this subsection is to characterize the fixed maximal ideals of $A$. First we need some preliminaries.
 
\begin{prop}\label{free0}  Let $f\in A$ whose zeroset $\Zz_X(f)$ is not compact. Then $f$ lies in some proper free ideal $\gta$ of $A$.
\end{prop}
\begin{proof} Since $\Zz_X(f)$ is not compact, there exists a family $\{W_i\}_{i\in I}$ of open semialgebraic subsets of $\R^n$ which covers $\Zz_X(f)$ and admitting no finite subcovering. For each index $i\in I$ there exists $g_i\in{\mathcal S}^*(\R^n)$ such that $\Zz_{\R^n}(g_i)=\R^n\setminus W_i$. Let us show that the ideal $\gta$ of $A$ generated by $f$ and the restrictions $f_i:=g_i|_X\in A$ is a proper free ideal. In case $\gta=A$ we have an equality 
\begin{equation}\label{total}
1=gf+\sum_{j\in J}f_jh_j
\end{equation} 
for some finite subset $J$ of $I$ and some functions $g,h_j\in A$. Since the finite family $\{W_j\}_{j\in J}$ does not cover $\Zz_X(f)$ there exists a point $x\in\Zz_X(f)\setminus\bigcup_{j\in J}W_j$, which contradicts the equality \eqref{total}. Thus, $\gta$ is a proper ideal of $A$, and we check now that it is free. Since  $\Zz_X(f)\subset\bigcup_{i\in I}W_i$ we have 
$$
\bigcap_{h\in\gta}\Zz_X(h)=\Zz_X(f) \cap \bigcap_{i\in I}\Zz_X(f_i)=\Zz_X(f)\cap\bigcap_{i\in I}(\R^n\setminus W_i)=\varnothing,
$$
and so $\gta$ is a proper free ideal of $A$.
\end{proof}

\begin{remarks}\label{arquimediano} \em (1) Let $X\subset\R^n$ be a semialgebraic set and let $A$ be an intermediate $\R$-algebra between ${\mathcal S}^{*}(X)$ and ${\mathcal S}(X)$. The semialgebraic homeomorphism
$$
\varphi:\BB_n(0,1):=\{x\in\R^n:\, \|x\|<1\}\to\R^n,\ x\mapsto\frac{x}{\sqrt{1-\|x\|^2}},
$$
induces an $\R$-algebras isomorphism ${\widetilde\varphi}:{\mathcal S}(X)\to{\mathcal S}(Y),\,f\mapsto f\circ\varphi$, where $Y:=\varphi^{-1}(X)$ is bounded, that maps ${\mathcal S}^*(X)$ onto ${\mathcal S}^*(Y)$. Thus, $B:={\widetilde\varphi}(A)$ is an intermediate $\R$-algebra between ${\mathcal S}^*(Y)$ and ${\mathcal S}(Y)$. Hence, substituting $X$ by $Y$ and $A$ by $B$ if necessary, we may always assume that $X$ is bounded.

\noindent (2) If $X$ is not compact then $A$ has free maximal ideals. Indeed, we may assume, by part (1), that $X$ is bounded, and so there exists a point $p\in\cl_{\R^n}(X)\setminus X$. Consider the bounded semialgebraic function $f:X\to\R,\ x\to\|x-p\|$ whose zeroset is empty, so the ideal generated by $f$ in $A$ is free. 

\noindent (3) The fixed maximal ideals of $A$ are those of the form $\gtm_{A;x}$ for some $x\in X$. We noticed in the proof of Corollary \ref{convex1} (5) that $\gtm_{A;x}$ is a maximal ideal. In addition it is fixed because $x\in\Zz_X(f)$ for every $f\in\gtm_{A;x}$. Conversely, let $\gtn\in\Max(A)$ be a fixed ideal and let $x\in X$ be a point such that all functions in $\gtn$ vanish at $x$. This means that $\gtn\subset\gtm_{A;x}$ and, since $\gtn$ is maximal, the equality $\gtn=\gtm_{A;x}$ holds true.

\noindent (4) Let $\gtn$ be a maximal ideal of ${\mathcal S}^*(X)$. The map $\R\to{\mathcal S}^*(X)/\gtn;\, r\mapsto r+\gtn$, is an isomorphism. It is injective because $\R$ is a field, and it is surjective because $\R$ does not admit proper archimedean extensions and ${\mathcal S}^*(X)/\gtn$ is an archimedean extension of $\R$ since given $f\in{\mathcal S}^*(X)$ there exists $r\in\R$ such that $|f(x)|<r$ for every $x\in X$. Thus, since $\R$ admits a unique automorphism, there is no ambiguity to refer $f+\gtn\in\R$ as a real number for every $f\in{\mathcal S}^*(X)$. In particular, for each $x\in X$ the isomorphism ${\mathcal S}^*(X)/\gtm_x^{*}\cong\R$ identifies $f+\gtm_x^{*}$ with $f(x)$.

\noindent (5) It was proved in \cite[Corollary 3.10]{fg2} that the equality ${\mathcal S}^*(X)={\mathcal S}(X)$ holds if and only if $X$ is compact. In this case there is no proper intermediate algebras between ${\mathcal S}^*(X)$ and ${\mathcal S}(X)$.

\noindent (6) In the statement of Theorem \ref{freefixed} and the proof of Proposition \ref{lojasiewicz} we will use \cite[3.6]{fg2} that says that for every maximal ideal $\gtm^*$ of ${\mathcal S}^{*}(X)$ there exists a unique maximal ideal $\gtm$ of ${\mathcal S}(X)$ such that $\gtm\cap{\mathcal S}^{*}(X)\subset\gtm^*$.
 
\end{remarks}

\begin{thm}\label{freefixed} Let $\gtm_A$ be a maximal ideal of $A$, let $\gtp:=\gtm_A\cap{\mathcal S}^{*}(X)$ and let $\gtm^{*}$ be the unique maximal ideal of ${\mathcal S}^{*}(X)$ containing $\gtp$. Let $\gtm$ be the unique maximal ideal of ${\mathcal S}(X)$ such that $\gtq:=\gtm\cap{\mathcal S}^{*}(X)\subset\gtm^*$. Then, the following assertions are equivalent:

\noindent  \em (1) $\hgt(\gtm_A)=\hgt(\gtm^*)=\hgt(\gtm)$.

\noindent (2) $\gtp=\gtm^*=\gtq$.

\noindent (3) $\gtm$ is a fixed ideal of ${\mathcal S}(X)$ and $\gtm_A$ is a fixed ideal of $A$. 

\noindent (4) $\gtm^*$ is a fixed ideal of ${\mathcal S}^{*}(X)$.
\end{thm}

\begin{proof} The equivalence (1) $\iff (2)$ follows from Proposition \ref{saturado} because $A={\mathcal S}^*(X)_{D_A}$ is the localization of ${\mathcal S}^*(X)$ at $D_A$ and ${\mathcal S}(X)$ is the localization of ${\mathcal S}^*(X)$ at the multiplicatively closed set ${\mathcal W}(X):=\{f\in{\mathcal S}^*(X): \Zz_X(f)=\varnothing\}$. 

\noindent (2) $\Longrightarrow$ (3) Since $\gtq=\gtm^*$ it follows from \cite[3.7]{fg2} that $\gtm$ is a fixed ideal of ${\mathcal S}(X)$. Let $x\in X$ such that all functions in $\gtm$ vanish at $x$. As $\gtp=\gtq\subset\gtm$, all functions in $\gtp$ vanish at $x$. Let $f\in\gtm_A$. Then $f\cdot(1+f^2)^{-1}\in\gtp$ vanish at $x$, and so $f(x)=0$. Thus, $\gtm_A$ is a fixed ideal of $A$.

\noindent (3) $\Longrightarrow (4)$ Let $x\in X$ such that $\gtm:=\gtm_{A;x}$. All functions in $\gtp$ vanish at $x$, so $\gtp\subset\gtm_x^{*}$. Since ${\mathcal S}^{*}(X)$ is a Gelfand ring, $\gtm^{*}=\gtm_x^{*}$ is a fixed ideal.

\noindent (4) $\Longrightarrow (2)$ The equality $\gtq=\gtm^*$ follows from \cite[3.7]{fg2}. To show that $\gtp=\gtm^*$ is equivalent to prove that $\gtp$ is a maximal ideal. To that end note that since $\gtm^*$ is a fixed ideal there exists a point $x\in X$ such that each function in $\gtm^*$ vanishes at $x$. Let $f\in\gtm_A$. Since $(1+f^2)^{-1}\in{\mathcal S}^{*}(X)\subset A$ the function $g:=f\cdot(1+f^2)^{-1}\in\gtp\subset\gtm^*$. Thus $g(x)=0$ and so $f(x)=0$. This shows that $\gtm_A=\gtm_{A;x}$, which implies $\gtp=\gtm_A\cap{\mathcal S}^{*}(X)=\gtm_{A;x}\cap{\mathcal S}^{*}(X)=\gtm^*$. Hence $\gtp$ is a maximal ideal, as wanted.  
\end{proof}

\begin{prodef}\label{arquimedianO} \em (1)  A maximal ideal $\gtm$ of $A$ is said to be \em archimedean \em if the natural inmersion $\R\hookrightarrow A/\gtm$ is surjective. 

\noindent (2) All maximal ideals of ${\mathcal S}^*(X)$ are archimedean.

\noindent (3) The ideal $\gtm$ is archimedean if and only if $\gtp:=\gtm\cap{\mathcal S}^{*}(X)$ is a maximal ideal.
\end{prodef}
\begin{proof} (2) This is part (4) in Remark \ref{arquimediano}.

\noindent (3) Notice that $\R\subset{\mathcal S}^{*}(X)/\gtp\subset A/\gtm$. Thus, if $\gtm$ is archimedean we have $\R\subset{\mathcal S}^{*}(X)/\gtp\subset\R$, so ${\mathcal S}^{*}(X)/\gtp=\R$ is a field. Hence $\gtp$ is a maximal ideal.

Conversely, suppose that $\gtp$ is a maximal ideal. Let $D_A:=\Uu(A)\cap{\mathcal S}^{*}(X)$. The quotient $A/\gtm$ is the ring of fractions of ${\mathcal S}^{*}(X)/\gtp$ with respect to the image $D_A(\gtp):=\pi(D_A)$ by the projection $\pi:{\mathcal S}^{*}(X)\to{\mathcal S}^{*}(X)/\gtp$. As the ideal $\gtp$ is maximal the quotient ${\mathcal S}^{*}(X)/\gtp$ is a field.  Hence it coincides with its localization $({\mathcal S}^{*}(X)/\gtp)_{D_A(\gtp)}=A/\gtm$. Therefore, using part (4) in Remark \ref{arquimediano} again, $A/\gtm={\mathcal S}^{*}(X)/\gtp=\R$.
\end{proof}

We are in position to present a \L ojasiewicz's inequality for some intermediate $\R$-algebras between ${\mathcal S}^{*}(X)$ and ${\mathcal S}(X)$. First we introduce new notations.

\begin{nots} \em (1)  Let $f\in A$. We denote $\ceros_{\Max(A)}(f):=\{\gtm\in\Max(A): f\in\gtm\}$.

\noindent (2) In the particular case that $A:={\mathcal S}^{*}(X)$ we denote $\ceros_{\betaa X}(f):=\{\gtm^*\in\betaa X: f\in\gtm^*\}$.
\end{nots}

\begin{prop}[\L ojasiewicz's inequality]\label{lojasiewicz} Let $A$ be an intermediate algebra whose maximal ideals  are archimedean.  Let $f,g\in A$ be such that $\ceros_{\Max(A)}(f)\subset\ceros_{\Max(A)}(g)$. Then there exists $h\in A$ and a positive integer $\ell$ such that $g^\ell=fh$. 
\end{prop}
\begin{proof} Consider the bounded semialgebraic functions
$$
f_1:=\frac{f}{(1+f^2)\cdot(1+g^2)}\,  \text{ and }\,  g_1:=\frac{g}{(1+f^2)\cdot(1+g^2)}.
$$ 
Let us prove that $\ceros_{\betaa X}(f_1)\subset\ceros_{\betaa X}(g_1)$. Let $\gtm^*\in\ceros_{\betaa X}(f_1)$. By \cite[Thm. 3.5]{fg2} there exists a unique maximal ideal $\gtm$ of ${\mathcal S}(X)$ such that $\gtm\cap{\mathcal S}^{*}(X)\subset\gtm^*$. As $\gtm\cap A$ is a prime ideal of $A$ there exists, by Corollary \ref{convex1} (4) a unique maximal ideal $\gtm_A$ of $A$ such that $\gtm\cap A\subset\gtm_A$. By hypothesis $\gtm_A$ is an archimedean maximal ideal and, by \ref{arquimedianO}, $\gtm_A\cap{\mathcal S}^{*}(X)$ is a maximal ideal. As $\gtp:=\gtm\cap{\mathcal S}^{*}(X)\subset\gtm\cap A\subset\gtm_A$ we have $\gtp\subset\gtm_A\cap{\mathcal S}^{*}(X)$. Thus, $\gtm^*$ and $\gtm_A\cap{\mathcal S}^{*}(X)$ are maximal ideals of the Gelfand ring ${\mathcal S}^{*}(X)$ containing the prime ideal $\gtp$ of ${\mathcal S}^{*}(X)$. Consequently, $\gtm_A\cap{\mathcal S}^{*}(X)=\gtm^*$ and so $\gtm_A\cap{\mathcal S}^{*}(X)\in\ceros_{\betaa X}(f_1)$. In particular $f_1\in\gtm_A$ and so $f\in\gtm_A$. Hence $g\in\gtm_A$. Then $g_1\in\gtm_A\cap{\mathcal S}^{*}(X)=\gtm^*$, that is, $\gtm^*\in\ceros_{\betaa X}(g_1)$. Now, by \cite[3.10]{fg6}, there exist $h_1\in{\mathcal S}^{*}(X)\subset A$ and a positive integer $\ell$ such that $g_1^{\ell}=f_1h_1$. denote  $F:=(1+f^2)\cdot(1+g^2)\in A$. Then 
$$
\frac{g^{\ell}}{F^{\ell}}=\left(\frac{f}{F}\right)\cdot h_1, \, \text{ and so, } g^{\ell}=f\cdot(F^{\ell-1}\cdot h_1).
$$
Therefore $h:=F^{\ell-1}\cdot h_1$ does the job.
\end{proof}

Our next goal is to obtain a Nullstellensatz for functions in the intermediate algebra $A$. First we need to introduce some terminology and basic results.  

\begin{defnsprop}\label{zAfiltros} \em (1) The family of all sets $\ceros_{\Max(A)}(f)$ for $f\in A$ is denoted by 
$$
\ceros_{\Max(A)}:=\{\ceros_{\Max(A)}(f):\ f\in A\}.
$$ 
A subset $\Ff$ of $\Pp(\ceros_{\Max(A)})$ is a \em $z_A$-filter \em on $X$ if it satisfies the following properties:

\noindent (1.1) $\varnothing\not\in\Ff$.

\noindent (1.2) Given $Z_1,Z_2\in\Ff$ then $Z_1\cap Z_2\in\Ff$.

\noindent (1.3) Given $Z\in\Ff$ and $Z'\in\ceros_{\Max(A)}$ such that $Z\subset Z'$ then $Z'\in\Ff$.

\noindent (2) Let $\Ff$ be a $z_A$-filter on $X$. We denote
$$
{\mathcal J}(\Ff):=\{f\in A:\ \ceros_{\Max(A)}(f)\in\Ff\},
$$ 
which is a proper ideal of $A$ satisfying $\ceros_{\Max(A)}[\J(\Ff)]=\Ff$. Indeed, by condition (1.1) the set ${\mathcal J}(\Ff)$ contains no unit of $A$. Let $f,g\in{\mathcal J}(\Ff)$. Then
$$
\ceros_{\Max(A)}(f+g)\supset\ceros_{\Max(A)}(f)\cap\ceros_{\Max(A)}(g)\in\Ff.
$$
Hence, $\ceros_{\Max(A)}(f+g)\in\Ff$, and therefore $f+g\in{\mathcal J}(\Ff)$. Now, let $f\in{\mathcal J}(\Ff)$ and $g\in A$. Then 
$$
\ceros_{\Max(A)}(fg)=\ceros_{\Max(A)}(f)\cup\ceros_{\Max(A)}(g)\supset\ceros_{\Max(A)}(f)\in\Ff.
$$
Thus, $\ceros_{\Max(A)}(fg)\in\Ff$, that is, $fg\in{\mathcal J}(\Ff)$. Finally, if $Z\in\ceros_{\Max(A)}[\J(\Ff)]$ there exists $f\in{\mathcal J}(\Ff)$ such that $Z=\ceros_{\Max(A)}(f)$, hence, $Z\in\Ff$. Consequently, $\ceros_{\Max(A)}[\J(\Ff)]=\Ff$. 

\noindent (3) An ideal $\gta$ of $A$ is said to be a \em $z_A$-ideal \em if $\J(\ceros_{\Max(A)}[\gta])=\gta$. Note  that each $z_A$-ideal is radical because $\ceros_{\Max(A)}(f)=\ceros_{\Max(A)}(f^k)$ for all $f\in A$ and all $k\geq 1$.

\noindent (4) Notice that the equality $\ceros_{\Max(A)}[\J(\Ff)]=\Ff$  implies that ${\mathcal J}(\Ff)$ is a $z_A$-ideal whenever $\Ff$ is a $z_A$-filter. 
\end{defnsprop}

\begin{example} \em  Let $\gta$ be a proper ideal of $A$. Let us check that: \em The family
$$
\ceros_{\Max(A)}[\gta]:=\{\ceros_{\Max(A)}(f):\ f\in\gta\}
$$ 
is a $z_A$-filter on $X$\em. Indeed, if $\varnothing\in\ceros_{\Max(A)}[\gta]$ there exists $f\in\gta$ such that $\ceros_{\Max(A)}(f)=\varnothing$, that is, $f\not\in\gtm$ for each maximal ideal $\gtm$ of $A$. Hence, $f$ is a unit in $A$, a contradiction.  In addition, let $f,g\in\gta$. Then,
$$
\ceros_{\Max(A)}(f)\cap\ceros_{\Max(A)}(g)=\ceros_{\Max(A)}(f^2+g^2)\in\ceros_{\Max(A)}[\gta]. 
$$
The inclusion 
$$
\ceros_{\Max(A)}(f)\cap\ceros_{\Max(A)}(g)=\ceros_{\Max(A)}(f^2+g^2)
$$ 
is evident. Conversely, let $\gtm\in\ceros_{\Max(A)}(f^2+g^2)$. Since $0\leq f^2(x)\leq f^2(x)+g^2(x)$ por every $x\in X$ it follows from Corollary \ref {convex1} (3) that $f^2\in\gtm$, and so $f\in\gtm$. Analogously, $g\in\gtm$. Thus $\gtm\in\ceros_{\Max(A)}(f)\cap\ceros_{\Max(A)}(g)$. Finally, let $\ceros_{\Max(A)}(g)\subset\ceros_{\Max(A)}(f)$ and $g\in\gta$. Then $\ceros_{\Max(A)}(f)=\ceros_{\Max(A)}(fg)\in\ceros_{\Max(A)}[\gta]$, because $fg\in\gta$.
\end{example}

We are ready to apply \L ojasiewicz inequality, \ref{lojasiewicz}, to prove the following Nullstellensatz for intermediate $\R$-algebras between ${\mathcal S}^{*}(X)$ and ${\mathcal S}(X)$.

\begin{cor}[Nullstellensatz]\label{radical} Let $A$ be an intermediate algebra whose maximal ideals are archimedean. Let $\gta$ be an ideal of $A$. Then, 

\noindent \em (1) \em $\gta$ is a $z_A$-ideal if and only if $\gta$ is radical. 

\noindent  \em (2) \em $\J(\ceros_{\Max(A)}[\gta])=\sqrt{\gta}$. 

\noindent  \em (3) \em In particular, if $\gtp$ is a prime ideal, then $\gtp$ is a $z_A$-ideal.
\end{cor}

\begin{proof} (1) Let $\gta$ be a radical ideal of $A$ and let $g\in A$ such that $\ceros_{\Max(A)}(g)\in\ceros_{\Max(A)}[\gta]$. Then, there exists $f\in\gta$ such that $\ceros_{\Max(A)}(f)=\ceros_{\Max(A)}(g)$. By Proposition \ref{lojasiewicz} there exist $\ell\geq 1$ and $h\in A$ such that $g^\ell=fh\in\gta$. Since $\gta$ is a radical ideal, $g\in\gta$. Thus, radical ideals are $z_A$-ideals, and the converse was noticed above in \ref{zAfiltros} (3).

\noindent (2) For each $g\in\J(\ceros_{\Max(A)}[\gta])$ there exists $f\in\gta$ such that $\ceros_{\Max(A)}(g)=\ceros_{\Max(A)}(f)$ and so, by Proposition \ref{lojasiewicz}, $g^\ell=fh\in\gta$ for some integer $\ell\geq1$ and some function $h\in A$. Hence, $g\in\sqrt{\gta}$. Since the argument is reversible we deduce $\J(\ceros_{\Max(A)}[\gta])=\sqrt{\gta}$.

\noindent (3) Since prime ideals are radical we conclude that all prime ideals of $A$ are $z_A$-ideals.
\end{proof}

\section{Intermediate algebras generated by proper ideals}\label{s4}

As in the precedent sections, $X\subset\R^n$ is a semialgebraic set and $A$ is an intermediate $\R$-algebra $A$ between ${\mathcal S}^*(X)$ and ${\mathcal S}(X)$

\begin{defn} \em Let $\Lambda\subset{\mathcal S}(X)$ be a family of semialgebraic functions on $X$. Let ${\mathcal S}^{*}(X)[{\tt x_{\Lambda}}]$ be the polynomial ring with coefficients in ${\mathcal S}^{*}(X)$ in the variables ${\tt x}_{f}$ where $f\in\Lambda$. Let 
$$
\ev:{\mathcal S}^{*}(X)[{\tt x_{\Lambda}}]\to{\mathcal S}(X)
$$
be the homomorphism that fixes each function in ${\mathcal S}^{*}(X)$ and maps each variable ${\tt x}_{f}$ to $f$. We denote its image ${\mathcal S}^{*}(X)[\Lambda]$ and we call it the intermediate $\R$-algebra \em generated by $\Lambda$ \em over ${\mathcal S}^{*}(X)$. Clearly it is the smallest subring of ${\mathcal S}(X)$ containing ${\mathcal S}^{*}(X)\cup\Lambda$. In case $\Lambda$ is finite we say that ${\mathcal S}^{*}(X)[\Lambda]$ is \em finitely generated. \em 

In particular, if $\Lambda:=\{f\}$ is a singleton we write ${\mathcal S}^{*}(X)[\Lambda]:={\mathcal S}^{*}(X)[f]$, and we say that ${\mathcal S}^{*}(X)[f]$ is a \em simple \em extension of ${\mathcal S}^{*}(X)$.
\end{defn}

\begin{prop}[Convexity II]\label{convex2} \em (1) \em The $\R$-algebra $A$ is \em absolutely convex, \em that is, whenever $f,g\in{\mathcal S}(X)$ satisfy $|f(x)|\leq|g(x)|$ for every $x\in X$, and $g\in A$, then $f\in A$.
  
\noindent \em (2) \em Let $f\in{\mathcal S}(X)$. Then, $f\in A$ if and only if $|f|\in A$.

\noindent \em (3) \em  The intermediate $\R$-algebras between ${\mathcal S}^*(X)$ and ${\mathcal S}(X)$ are, exactly, the absolutely convex subrings of ${\mathcal S}(X)$ containing the constant functions. 
\end{prop}
\begin{proof} (1) The product $h:=f\cdot(1+g^2)^{-1}\in{\mathcal S}^*(X)$ because $|f(x)|\leq|g(x)|$ for each point $x\in X$. Thus $h\in A$ and so $f=h\cdot(1+g^2)\in A$ too.

\noindent (2) This is the immediate consequence of part (1). 

\noindent (3) Let $A$ be an absolutely convex subring of ${\mathcal S}(X)$ containing the constant functions. Let $f\in{\mathcal S}^*(X)$ and let $r\in\R$ be such that $|f(x)|\leq r$ for every $x\in X$.  As $A$ is absolutely convex we deduce that $f\in A$.
\end{proof}

\begin{prop}  \em (1) \em Let $A$ be a simple extension of ${\mathcal S}^{*}(X)$. Then, for every real number $r>1$ there exists $f\in{\mathcal S}(X)$ such that $f(x)\geq r$ for every point $x\in X$ and $A={\mathcal S}^{*}(X)[f]$.

\noindent \em (2) \em Let $f\in{\mathcal S}(X)$ be such that there exists a real number $r>1$ satisfying $f(x)\geq r$ for each $x\in X$. Then,
$$
{\mathcal S}^{*}(X)[f]=\{g\in{\mathcal S}(X): \text{ there exists  } k\in\Z^{+} \text{ such that }\, |g(x)|\leq f(x)^k \, \text{ for each point }\, x\in X\}.
$$
\noindent \em (3) \em Let $g_1,\dots,g_n\in{\mathcal S}(X)$. Then,
$$
{\mathcal S}^{*}(X)[g_1,\dots,g_n]={\mathcal S}^{*}(X)[|g_1|+\cdots+|g_n|].
$$
In particular, every finitely generated intermediate $\R$-algebra between ${\mathcal S}^{*}(X)$ and ${\mathcal S}(X)$ is simple.
\end{prop}
\begin{proof} (1) Let $g\in A$ such that $A={\mathcal S}^{*}(X)[g]$. Since $g\in A$ it follows from Proposition \ref{convex2} that $|g|\in A$. As the constant function $r\in{\mathcal S}^{*}(X)\subset A$, the sum $f:=|g|+r\in A$. Consequently ${\mathcal S}^{*}(X)[f]\subset A$ and $f(x)\geq r$ for every point $x\in X$. To prove the converse inclusion note that since ${\mathcal S}^{*}(X)[f]$ is an intermediate $\R$-algebra between ${\mathcal S}^{*}(X)$ and ${\mathcal S}(X)$ it follows from Proposition \ref{convex2} that ${\mathcal S}^{*}(X)[f]$ is absolutely convex. In addition, $|g(x)|<f(x)=|f(x)|$ for each $x\in X$ and $f\in{\mathcal S}^{*}(X)[f]$. Therefore $g\in{\mathcal S}^{*}(X)[f]$. Thus $A={\mathcal S}^{*}(X)[g]\subset{\mathcal S}^{*}(X)[f]$ and the equality $A={\mathcal S}^{*}(X)[f]$ holds true.

\noindent (2) Since ${\mathcal S}^{*}(X)[f]$ is an absolutely convex intermediate $\R$-algebra between ${\mathcal S}^*(X)$ and ${\mathcal S}(X)$ that contains $f^k$ for every $k\in\Z^{+}$, the inclusion $\supset$ follows. Conversely, let $g\in {\mathcal S}^{*}(X)[f]$. Then there exists a nonnegative integer $n$ and $g_0,\dots,g_n\in{\mathcal S}^{*}(X)$ such that
$$
g(x)=\sum_{j=0}^ng_j(x)f(x)^j\, \text{ for every } x\in X.
$$
Let $c>1$ be a real number such that $|g_j(x)|<c$ for every $x\in X$ and $j=0,\dots,n$. Note that $\lim_{m\to\infty}\{r^m\}=+\infty$, so there exists $m_0\in\N$ such that $n+1,c<r^{m_0}\leq f(x)^{m_0}$ for every $x\in X$. Thus, the integer $k:=2m_0+n$ satisfies that for every $x\in X$,
$$
|g(x)|\leq\sum_{j=0}^n|g_j(x)|\cdot f(x)^j\leq f(x)^{m_0}\cdot\sum_{j=0}^nf(x)^n=(n+1)\cdot f(x)^{m_0}\cdot f(x)^{n}\leq f(x)^{2m_0+n}=f(x)^k.
$$ 
\noindent (3) Define $f:=|g_1|+\cdots+|g_n|$ and $A:={\mathcal S}^{*}(X)[g_1,\dots,g_n]$. Note that for each point $x\in X$ and each $j=1,\dots,n$ we have $|g_j(x)|\leq|g_1(x)|+\cdots+|g_n(x)|=|f(x)|$. Thus, by Proposition \ref{convex2} (1), each $g_j\in{\mathcal S}^{*}(X)[f]$. Hence, $A\subset{\mathcal S}^{*}(X)[f]$. For the converse inclusion note that, by Proposition \ref{convex2} (2), each $|g_j|\in A$, so $f\in A$ and ${\mathcal S}^{*}(X)[f]\subset A$.
\end{proof}

In \ref{zAfiltros} we studied filters of zerosets in the space of maximal ideals of an intermdiate $\R$-algebra between ${\mathcal S}^{*}(X)$ and ${\mathcal S}(X)$. Now we need similar results about filters of closed semialgebraic sets.

\subsection{Filters of closed semialgebraic subsets.}\label{Filters} Let $X\subset\R^n$ be a semialgebraic set and let $\|\cdot\|$ be the euclidean norm of $\R^n$. Let $Z$ be a closed semialgebraic subset of $X$. It was proved in \cite[2.2.8]{bcr} that the function
$$
f:X\to\R,\, x\mapsto\dist(x,Z)=\inf\, \{\|x-z\|:\, z\in Z\}
$$
is semialgebraic and, since $Z$ is closed in $X$, we have $Z=\Zz_X(f)$. Thus every closed semialgebraic subset of $X$ is the zeroset of a continuous semialgebraic function defined on $X$. 

\noindent (3) \label{zfilter} Let $\Zz_X$ be the collection of all closed semialgebraic subsets of $X$. Let $\Pp(\Zz_X)$ be the set of all subsets of $\Zz_X$. A subset $\Ff$ of $\Pp(\Zz_X)$ is a \em semialgebraic filter \em on $X$ if it satisfies the following properties:

\noindent (3.1) $\varnothing\not\in\Ff$.

\noindent (3.2) Given $Z_1,Z_2\in\Ff$ then $Z_1\cap Z_2\in\Ff$.

\noindent (3.3) Given $Z_1\in\Ff$ and $Z_2\in\Zz_X$ such that $Z_1\subset Z_2$ then $Z_2\in\Ff$.

\noindent (4) Let $\Ff$ be a semialgebraic filter on $X$. We define the \em closure \em of $\Ff$ in $\betaa X$ as the set
$$
\cl_{\betaa X}(\Ff):=\bigcap_{Z\in\Ff}\cl_{\betaa X}(Z).
$$
\noindent (5) Let $\gta$ be a proper ideal of ${\mathcal S}(X)$. Then:

\noindent (5.1) The family $\Zz_X[\gta]:=\{\Zz_X(f):\, f\in\gta\}$ is a semialgebraic filter on $X$.

The units of ${\mathcal S}(X)$ are those semialgebraic functions $f$ with empty zeroset since in such a case  $X\to \R,\, x\mapsto 1/f(x)$ is a well defined semialgebraic function. Thus $\varnothing\not\in\Zz_X[\gta]$ because $\gta$ is a proper ideal of ${\mathcal S}(X)$. In addition, if $f,g\in\gta$ satisfy $Z_1:=\Zz_X(f)$ and $Z_2:=\Zz_X(g)$ then $f^2+g^2\in\gta$ and $Z_1\cap Z_2=\Zz_X(f^2+g^2)\in\Zz_X(\gta)$. Finally, let $Z_1\in\Zz_X(\gta)$ and $Z_2\in\Zz_X$ be such that $Z_1\subset Z_2$. By (2) there exists $f\in{\mathcal S}(X)$ such that $Z_2=\Zz_X(f)$. Since $Z_1\in\Zz_X(\gta)$ there exists $g\in\gta$ such that $\Zz_1=\Zz_X(g)$. Thus, $h=fg\in\gta$ and $Z_2=\Zz_X(h)\in\Zz_X[\gta]$.

\noindent (5.2) If $\Ff$ is a semialgebraic filter on $X$, then ${\mathcal J}(\Ff):=\{f\in{\mathcal S}(X):\, \Zz_X(f)\in\Ff\}$ is a proper ideal of ${\mathcal S}(X)$ satisfying $\Zz_X[\J(\Ff)]=\Ff$.

Indeed, given $f,g\in{\mathcal J}(\Ff)$ their zero sets $Z_1:=\Zz_X(f)$ and $Z_2:=\Zz_X(g)$ belong to $\Ff$. Thus 
$$
\Zz_X(f^2+g^2)=\Zz_X(f)\cap\Zz_X(g)=Z_1\cap Z_2\in\Ff
$$
and $Z_1\cap Z_2\subset\Zz_X(f-g)$, so $\Zz_X(f-g)\in\Ff$, that is, $f-g\in{\mathcal J}(\Ff)$. Furthermore, given $f\in{\mathcal S}(X)$ and $g\in{\mathcal J}(\Ff)$ we have $\Zz_X(g)\in\Ff$ and $\Zz_X(g)\subset\Zz_X(fg)$. Thus  $\Zz_X(fg)\in\Ff$ and so $fg\in{\mathcal J}(\Ff)$. This proves that ${\mathcal J}(\Ff)$ is an ideal of ${\mathcal S}(X)$, and it is proper because $\Zz_X(1)=\varnothing\notin\Ff$.

Let us check the equality $\Zz_X[\J(\Ff)]=\Ff$. Given $Z\in\Ff$ there exists, by (2), $f\in{\mathcal S}(X)$ such that $Z=\Zz_X(f)$. Hence $f\in\J(\Ff)$ and $Z=\Zz_X(f)\in\Zz_X[\J(\Ff)]$. Conversely, let $Z\in\Zz_X[\J(\Ff)]$. Then there exists $f\in\J(\Ff)$ such that $Z=\Zz_X(f)\in\Ff$. 

\noindent (6) An ideal $\gta$ of ${\mathcal S}(X)$ is a \em $z$-ideal \em if $\J(\Zz_X[\gta])=\gta$, that is, whenever there exist $f\in\gta$ and $g\in{\mathcal S}(X)$ satisfying $\Zz_X(f)\subset\Zz_X(g)$, we have $g\in\gta$.

\begin{remark}\label{des0} \em Notice that the equality $\Zz_X[\J(\Ff)]=\Ff$ implies that ${\mathcal J}(\Ff)$ is a $z$-ideal whenever $\Ff$ is a semialgebraic filter, because given $f\in{\mathcal J}(\Ff)$ and $g\in{\mathcal S}(X)$ satisfying $\Zz_X(f)\subset\Zz_X(g)$ it follows that $\Zz_X(g)\in\Ff$ since $\Ff$ is a semialgebraic filter and $\Zz_X(f)\in\Ff$. Hence $g\in\J(\Ff)$. Thus, every semialgebraic filter $\Ff$ on $X$ has the form $\Ff=\Zz_X[\gta]$ for some $z$-ideal $\gta$ of ${\mathcal S}(X)$. 
\end{remark}

\subsection{Intermediate $\R$-algebras generated by proper ideals}

\begin{defn} \em  The intermediate $\R$-algebra between ${\mathcal S}^{*}(X)$ and ${\mathcal S}(X)$ \em generated \em  by the proper ideal $\gta$ of ${\mathcal S}(X)$ is 
$$
A:={\mathcal S}^{*}(X)+\gta:=\{f+g:f\in{\mathcal S}^{*}(X),\, g\in\gta\}.
$$ 
\end{defn}

We shall see in Proposition \ref{zclosure} that in the definition above we may assume that $\gta$ is a $z$-ideal. First we need two auxiliary lemmas.

\begin{lem} Let $\gta$ be an ideal of ${\mathcal S}(X)$. The \em $z$-closure \em of $\gta$ is the set
$$
{\ol\gta}^{z}:=\{f\in{\mathcal S}(X): \text{ there exists }\, g\in\gta\, \text{ such that }\, \Zz_X(g)=\Zz_X(f)\},
$$
which is the smallest $z$-ideal of ${\mathcal S}(X)$ containing $\gta$.
\end{lem}
\begin{proof} Let us see that ${\ol\gta}^{z}$ is an ideal of ${\mathcal S}(X)$.  First we prove that $f:=f_1-f_2\in{\ol\gta}^{z}$ for every $f_1,f_2\in{\ol\gta}^{z}$. Let $g_1,g_2\in\gta$ such that $\Zz_X(g_i)=\Zz_X(f_i)$ for $i=1,2$. Then $g:=g_1^2+g_2^2\in\gta$, so $fg\in\gta$ and 
\begin{equation*}
\begin{split}
\Zz_X(fg)&=\Zz_X(f)\cup\Zz_X(g)=\Zz_X(f)\cup(\Zz_X(g_1)\cap\Zz_X(g_2))=\Zz_X(f)\cup(\Zz_X(f_1)\cap\Zz_X(f_2))\\
&\subset\Zz_X(f)\cup\Zz_X(f_1-f_2)=\Zz_X(f)\subset\Zz_X(fg) ,
\end{split}
\end{equation*}
so $\Zz_X(f)=\Zz_X(fg)$, which implies that $f\in{\ol\gta}^{z}$. In addition, let $\ell\in{\ol\gta}^{z}$ and $h\in{\mathcal S}(X)$. There exists $g\in\gta$ with $\Zz_X(g)=\Zz_X(\ell)$, so $hg\in\gta$ and 
$$
\Zz_X(hg)=\Zz_X(h)\cup\Zz_X(g)=\Zz_X(h)\cup\Zz_X(\ell)=\Zz_X(h\ell).
$$ 
Thus $hf\in{\ol\gta}^{z}$. To show that ${\ol\gta}^{z}$ is a $z$-ideal note that given $f,g\in{\mathcal S}(X)$ with $f\in{\ol\gta}^{z}$ and $\Zz_X(f)\subset\Zz_X(g)$ there exists $h\in\gta$ such that $\Zz_X(f)=\Zz_X(h)$. Hence, $hg\in\gta$ and
$$
\Zz_X(hg)=\Zz_X(h)\cup\Zz_X(g)=\Zz_X(f)\cup\Zz_X(g)=\Zz_X(g)\subset\Zz_X(hg), 
$$
that is, $\Zz_X(g)=\Zz_X(hg)$ and $hg\in\gta$. Therefore, $g\in{\ol\gta}^{z}$. 

Finally, the inclusion $\gta\subset{\ol\gta}^{z}$ is evident and, if $\gtb$ is a $z$-ideal containing $\gta$ it also contains ${\ol\gta}^{z}$. In fact, for every $f\in{\ol\gta}^{z}$ there exists $g\in\gta$ such that $\Zz_X(f)=\Zz_X(g)$. Since $\gtb$ is a $z$-ideal containing $g$ it follows that $f\in\gtb$.
\end{proof}

\begin{lem}\label{quot1} Let $f,g\in{\mathcal S}(X)$ with $\Zz_X(g)\subset\Int_X(\Zz_X(f))$. Then, there exists $h\in{\mathcal S}(X)$ with $f=gh$ and $\Zz_X(f)\subset\Zz_X(h)$. 
\end{lem}
\begin{proof} The open semialgebraic subsets $U:=X\setminus\Zz_X(g)$ and $V:=\Int_X(\Zz_X(f))$ of $X$ cover $X$ and $f|_{U\cap V}\equiv0$. Thus, 
$$
h:X\to\R,\ x\mapsto\left\{
\begin{array}{cc}
\frac{f(x)}{g(x)}&\text{if $x\in U$,}\\[4pt]
0&\text{if $x\in V$,}\\
\end{array}
\right.
$$
is a semialgebraic function satisfying $f=gh$ and $\Zz_X(f)\subset\Zz_X(h)$.
\end{proof}

\begin{prop}\label{zclosure} Let $\gta$ be a proper ideal of ${\mathcal S}(X)$. Then
$$
{\mathcal S}^{*}(X)+\gta={\mathcal S}^{*}(X)+{\ol\gta}^{z}.
$$
Hence the intermediate $\R$-algebra generated by an ideal is also generated by a $z$-ideal.
\end{prop}

\begin{proof}The inclusion $\subset$ follows because $\gta\subset{\ol\gta}^{z}$. Conversely, let $f:=g+h$ where $g\in{\mathcal S}^{*}(X)$ and $h\in{\ol\gta}^{z}$. Consider the disjoint closed semialgebraic subsets of $X$ 
$$
C_1:=\{x\in X:\, |h(x)|\leq1/2\} \, \text{ and }\, C_2:=\{x\in X:\, |h(x)|\geq1\}.
$$
By \cite[Cor. 2.4]{fg2} there exists $f\in{\mathcal S}^{*}(X)$ such that $f|_{C_1}\equiv0$ and $f|_{C_2}\equiv1$. In addition, after substituting $f$ by the quotient 
$$
\frac{2|f|}{1+f^2}
$$
we may assume that $0\leq f(x)\leq1$ for every point $x\in X$. Define $h_1:=f\cdot h$ and let us show that the semialgebraic function $h-h_1$ is bounded. Pick a point $x\in X$. If $|h(x)|\geq1$ then $x\in C_2$ and so $f(x)=1$. Thus $h_1(x)=h(x)$ and $(h-h_1)(x)=0$. On the other hand, if $|h(x)|\leq1$ then
$$
|(h-h_1)(x)|=|h(x)|\cdot|1-f(x)|\leq1
$$ 
because $0\leq f(x)\leq1$. Hence, $h-h_1\in{\mathcal S}^{*}(X)$, as wanted. In addition, as $h\in{\ol\gta}^{z}$ there exists $\ell\in\gta$ such that $\Zz_X(h)=\Zz_X(\ell)$. Consider the open semialgebraic subset
$$
U:=\{x\in X: |h(x)|<1/2\}. 
$$ 
Then, 
$$
\Zz_X(\ell)=\Zz_X(h)\subset U\subset C_1\subset\Zz_X(f)\subset\Zz_X(h_1),
$$
which implies $\Zz_X(\ell)\subset\Int_X(\Zz_X(h_1))$.  By Lemma \ref{quot1} there exists $h_2\in{\mathcal S}(X)$ such that $h_1=h_2\cdot\ell$. Finally,
$$
f=(f-h_1)+h_2\cdot\ell=(g+h-h_1)+h_2\cdot\ell\in{\mathcal S}^{*}(X)+\gta
$$
because both $g$ and $h-h_1$ are bounded and $h_2\cdot\ell\in\gta$.
\end{proof}

\begin{defn} \em Let $\Ff$ be a semialgebraic $z$-filter on the semialgebraic set $X$. 

\noindent (1) Given a semialgebraic subset $Y\subset X$ we define
$$
{\mathcal S}^{*}(X\,|\,Y):=\{f\in{\mathcal S}(X) \text{ such that } f|_{Y} \text{ is bounded}\}.
$$ 
(2) We denote
$$
{\mathcal S}^{*}(X\,|\,\Ff):=\bigcup_{Z\in\Ff}{\mathcal S}(X\,|\,Z).
$$
\end{defn}

\begin{prop}\label{acotadas} \em  Let $\gta$ be a proper ideal of ${\mathcal S}(X)$. Then ${\mathcal S}^{*}(X)+\gta={\mathcal S}^{*}(X\,|\,\Zz_X[\gta])$.
\end{prop}
\begin{proof} Let $f\in{\mathcal S}^{*}(X)+\gta$. Then there exist $g\in{\mathcal S}^{*}(X)$ and $h\in\gta$ such that $f=g+h$. Thus $f\equiv g$ on $\Zz_X(h)$, and so $f$ is bounded on $\Zz_X(h)\in\Zz_X[\gta]$, i.e. $f\in{\mathcal S}^{*}(X\,|\,\Zz_X[\gta])$.

Conversely, let $f\in{\mathcal S}(X)$ that is bounded on the set $\Zz_X(h)$ for some function $h\in\gta$. Thus 
$f|_{\Zz_X(h)}:\Zz_X(h)\to\R$ is a bounded semialgebraic function on the closed semialgebraic subset $\Zz_X(h)$ of $X$. By the semialgebraic version of the Tietze--Urysohn Lemma due to Delfs and Knebusch, see \cite{dk}, there exists $f_1\in{\mathcal S}^{*}(X)$ such that $f_1|_{\Zz_X(h)}=f|_{\Zz_X(h)}$. Consequently, $\Zz_X(h)\subset\Zz_X(f-f_1)$. Hence $p:=h\cdot(f-f_1)\in\gta$ and $\Zz_X(f-f_1)=\Zz_X(p)$, which implies that $f-f_1\in\ol{\gta}^{z}$. Therefore, by Proposition \ref{zclosure},
$$
f=f_1+(f-f_1)\in{\mathcal S}^{*}(X)+\ol{\gta}^{z}={\mathcal S}^{*}(X)+\gta.
$$
\end{proof}

\begin{thm}\label{EXT} Let $\Ff$ be a semialgebraic filter on $X$. Then:

\noindent \em (1) \em There exists a $z$-ideal $\gta$ of ${\mathcal S}(X)$ such that  $\Ff=\Zz_X[\gta]$.

\noindent \em (2) \em For every $f\in{\mathcal S}^{*}(X)+\gta$ there exists a continuous function $F:X\cup\cl_{\betaa X}(\Ff)\to\R$ such that $F|_X=f$.

\noindent \em (3) \em Let $\gtm$ be a maximal ideal of ${\mathcal S}(X)$. Then, for every $f\in{\mathcal S}^{*}(X)+\gtm$ there exists a continuous function $F:X\cup\{\gtm\}\to\R$ such that $F|_X=f$.
\end{thm}
\begin{proof} (1) This has been proved in \ref{Filters} (5.2) and Remark \ref{des0}.  

\noindent (2) By part (1) and Proposition \ref{acotadas},
$$
{\mathcal S}^{*}(X\,|\,\Ff)={\mathcal S}^{*}(X\,|\,\Zz_X[\gta])={\mathcal S}^{*}(X)+\gta.
$$ 
Thus there exists $Z\in\Ff$ such that the restriction $f|_{Z}$ is bounded. Hence, by \ref{stone1} (4), there exists a unique continuous extension ${\widehat f}:\betaa Z\to\R$ of $f|_{Z}$. Let ${\varphi}:{\mathcal S}^{*}(X)\to{\mathcal S}^{*}(Z),\, f\mapsto f|_{Z}$. It was proved in \cite[6.3-5]{fg3} that the map
$$
{\widetilde{\varphi}}:\betaa Z\to\cl_{\betaa X}(Z),\, \gtp\mapsto{\varphi}^{-1}(\gtp)
$$
is a homeomorphism. Therefore, with an obvious abuse of notation, there exists a continuous extension ${\widehat f}:\cl_{\betaa X}(Z)\to\R$ of $f|_{Z}$. In addition $X\cap\cl_{\betaa X}(Z)=\cl_{X}(Z)=Z$ and ${\widehat f}|_{Z}=f|_Z$. Thus the function
$$
G:X\cup\cl_{\betaa X}(Z)\to\R,\, x\mapsto\left\{ \begin{array}{cl}  f(x) & \mbox{ if } x\in X, \\[2pt]   {\widehat f}(x)& \mbox{ if } x\in\cl_{\betaa X}(Z),\\ 
\end{array} \right.  
$$
is a well defined continuous function and $F:=G|_{X\cup\cl_{\betaa X}(\Ff)}$ satisfies $F|_X=f$.

\noindent (3) It is enough to apply part (2) to the ideal $\gta:=\gtm$. 
\end{proof}

\begin{com} \em It is a natural question to ask if a given intermediate $\R$-algebra $A$ containing ${\mathcal S}^*(X)$ and contained in ${\mathcal S}(X)$ for some semialgebraic set $X$ is isomorphic to ${\mathcal S}(Y)$ for some semialgebraic set $Y$. To provide examples in which the answer is negative we introduce right now the notion of \em real closed ring, \em see \cite[Definition 1]{cd}.
\end{com}

\begin{defn} \em A commutative ordered ring $A$ with unit that is not a field is said to be \em real closed \em if given $a,b\in A$ with $a<b$ and a polynomial $p\in A[\t]$ such that $p(a)\cdot p(b)<0$ then there exists $c\in A$ such that $a<c<b$ an $p(c)=0$.
\end{defn}

It was proved by N. Schwartz in \cite[§III.1]{s} that for every semialgebraic set $X$ the rings ${\mathcal S}^*(X)$ and ${\mathcal S}(X)$ partially ordered by: $f\leq g$ if $f(x)\leq g(x)$ for every point $x\in X$, are real closed rings.

\begin{examples} \em (1) Let $f:\R\to\R,\, x\mapsto |x|$, which is a continuous semialgebraic function. Then, the simple extension $A:={\mathcal S}^*(\R)[f]$ is not isomorphic to either ${\mathcal S}^*(Y)$ or ${\mathcal S}(Y)$ for every semialgebraic set $Y$ because $A$ is not a real closed ring. Indeed the polynomial $p(\t):=\t^2-f$ satisfies $p(0)=-f<0$ and $p(1+f)=f^2+f+1>0$, but $p$ has no root in $A$ since the semialgebraic function $g(x):=\sqrt{|x|}$ does not belong to $A$.

\noindent (2) Consider the continuous semialgebraic function
$$
f:\R\to\R,\ x\mapsto\left\{
\begin{array}{cc}
0&\text{if $x\leq0$,}\\[4pt]
x&\text{if $x\geq0$,}\\
\end{array}
\right.
$$
and the ideal $\gta:=f\cdot{\mathcal S}(\R)$. Then, the intermediate algebra $A:={\mathcal S}^*(\R)+\gta$ is not isomorphic to either ${\mathcal S}^*(Y)$ or ${\mathcal S}(Y)$ for every semialgebraic set $Y$ because $A$ is not a real closed ring. Indeed, the polynomial $p(\t):=\t^2-f$ satisfies $p(0)=-f<0$ and $p(1+f)=f^2+f+1>0$, but $p$ has no root in $A$. Suppose, by way of contradiction that there exists $g\in{\mathcal S}^*(\R)$ and $h\in{\mathcal S}(\R)$ such that $p(g+fh)=0$. Thus, $g(t)+t\cdot h(t)=\sqrt{t}$ for every real number $t\geq0$, that is,
$$
g(t)=\sqrt{t}\cdot(1+\sqrt{t}\cdot h(t))\ \text{ for every real number }\ t>0,
$$
which is a contradiction since $g$ is bounded.

\noindent (3) In the setting of classical rings of continuous functions it is known that if $\gtm$ is a non-real maximal ideal in ${\mathcal C}(X)$, then ${\mathcal C}(X)$ cannot be obtained by adjoining countably many elements to the subalgebra $\R+\gtm$. We have not been able to prove or disprove this result in the semialgebraic setting. One of the reasons is that semialgebraic algebra has a finitary flavour and known strategies do not adapts well to infinitary statements.

\noindent (4) It is easier to find a semialgebraic set $X$ and an ideal $\gta$ of ${\mathcal S}(X)$ such that $\R+\gta$ does not contain ${\mathcal S}^*(X)$. To that end it suffices to choose $X:=\R$ and the ideal $\gta$ in (2). The continuous semialgebraic function
$$
g:\R\to\R,\, x\mapsto \frac{1}{1+x^2}
$$
is bounded but it does not belong to $\R+\gta$. Otherwise there would exist $r\in\R$ and $h\in{\mathcal S}(\R)$ such that
$$
\frac{1}{1+x^2}=g(x)=r+f(x)\cdot h(x).
$$
Evaluating in $x:=0$ it follows $r=1$. Therefore,
$$
f(x)\cdot h(x)=\frac{1}{1+x^2}-1=\frac{-x^2}{1+x^2}
$$
and we get a contradiction evaluating at $x:=-1$ because $f(-1)=0$. 
\end{examples}

\noindent{\bf Acknowledgments} The authors are grateful to the anonymous referee for pointing out many interesting remarks and comments and, in particular, for encourage us to find the last examples in the paper.

\bibliographystyle{amsalpha}

\begin{thebibliography}{BCR}


\bibitem[AM]{am} M.F. Atiyah, I.G. Macdonald: Introduction to Commutative Algebra. \em Addison-Wesley Publishing Company\em , Inc. Massachussets: 1969.

\bibitem[BCR]{bcr} J. Bochnak, M. Coste, M.F. Roy: Real algebraic geometry. \em Ergeb. Math. \em  {\bf 36}, Springer-Verlag, Berlin: 1998.

\bibitem[BFG]{bfg} E. Baro, J.F. Fernando, J.M. Gamboa: Rings of differentiable semialgebraic functions. \em Selecta Math. \em (N.S)  {\bf 30} (2024) no. 4, paper no. 71, 56 pp.

\bibitem[CD]{cd} G. Cherlin,  M. Dickmann,  Real closed rings. I. Residue rings of rings of continuous functions. \em Fund. Math.\em  {\bf 126} (1986), no. 2, 147--183.

\bibitem[DK]{dk} H. Delfs, M. Knebusch: Separation, Retractions and homotopy extension in semialgebraic spaces. \em Pacific J. Math. \em {\bf114} (1984), no. 1, 47--71.

\bibitem[DGM]{dgm} J.M. Dom\'inguez, J. G\'omez, M.A. Mulero: Intermediate algebras between ${\mathcal C}^*(X)$ and ${\mathcal C}(X)$ as ring of fractions of ${\mathcal C}^*(X)$. \em Topology and its Applications. \em {\bf 77}  (1997), 115--130.

\bibitem[FG1]{fg3} J.F. Fernando, J.M. Gamboa: On the spectra of rings of semialgebraic functions. \em Collectanea Math. {\bf 63} \em (2012) no. 3, 299--331.

\bibitem[FG2]{fg6} J.F. Fernando, J.M. Gamboa: On \L ojasiewicz's inequality and the Nullstellensatz for rings of semialgebraic functions. \em Journal of Algebra. {\bf 399}, \em (2014), 475--488 .

\bibitem[FG3]{fg1} J.F. Fernando, J.M. Gamboa: On the Krull dimension of rings of semialgebraic functions. {\em Rev. Mat. Iberoam.} {\bf 31} (2015), no 3, 753-766. 

\bibitem[FG4]{fg2} J.F. Fernando, J.M. Gamboa: On the semialgebraic Stone-\v{C}ech compactification of a semialgebraic set. \em Trans. AMS \em  {\bf 364} (2012) no 7, 3479-3511.

\bibitem[GJ]{gj} L. Gillman, M. Jerison: Rings of continuous functions. {\em The Univ. Series in Higher Nathematics} {\bf 1}, D. Van Nostrand Company, Inc.: (1960).

\bibitem[MO]{dmo} G. De Marco, A. Orsatti: Commutative rings in which every prime ideal is contained in a unique maximal ideal. {\em Proc. Amer. Math. Soc.} {\bf 30} (1971), no 3, 459-466.

\bibitem[RS]{rs} B. Requejo, J.B. Sancho: Localizations in the rings of continuous functions. \em Topology and its Applications. \em {\bf 57} (1994), 87--93. 

\bibitem[S]{s} N. Schwartz: The basic theory of real closed spaces. \em Mem. Amer. Math. Soc. {\bf 77} \em (397),  (1989). 

\end{thebibliography}

\end{document}